%% file: chi_square.tex
\documentclass[11pt]{scrartcl}
\usepackage[english]{babel}
  \expandafter\let\csname equation*\endcsname\relax
  \expandafter\let\csname endequation*\endcsname\relax
\include{latex-template/symphony}

\usepackage{babelbib}
\selectbiblanguage{english}

\begin{document}
%\title[Linear Combination of $\chi^2$ Random Variables in String Theory]{On the Efficient Calculation of a Linear Combination of Chi-Square Random Variables with an Application in Counting String Vacua}
\title{On the Efficient Calculation of a Linear Combination of Chi-Square Random Variables with an Application in Counting String Vacua}

\author{Johannes Bausch\thanks{Department of Physics, Cornell University, Ithaca, NY 14853, USA} \footnote{Present address:
DAMTP, Centre for Mathematical Sciences, University of Cambridge, Wilberforce Road, Cambridge CB3 0WB, UK}}

\date{\today}

\maketitle

\begin{abstract}
Linear combinations of chi square random variables occur in a wide range of fields. Unfortunately, a closed, analytic expression for the pdf is not yet known.
Starting out from an analytic expression for the density of the sum of two gamma variables, a computationally efficient algorithm to numerically calculate the linear combination of chi square random variables is developed. An explicit expression for the error bound is obtained.
The proposed technique is shown to be computationally efficient, i.e. only polynomial in growth in the number of terms compared to the exponential growth of most other methods. It provides a vast improvement in accuracy and shows only logarithmic growth in the required precision. In addition, it is applicable to a much greater number of terms and currently the only way of computing the distribution for hundreds of terms.
As an application, the exponential dependence of the eigenvalue fluctuation probability of a random matrix model for 4d supergravity with N scalar fields is found to be of the asymptotic form exp(-0.35N).
\end{abstract}
%\pacs{02.60.Gf, 02.60.-x, 04.65.+e}
%\ams{41-04, 60A99, 62E17, 83E50}
%\submitto{\JPA}
%\maketitle

\section{Review and Main Results}
\subsection{Introduction}
Linear combinations of $\chi^2$ random variables occur in many different fields, in statistical hypothesis testing as well as in high energy physics. In string theory, they have had a relative recent comeback, occurring frequently in the study of random matrix models, such as the spectral analysis of the \knownth{Wishart} ensemble. Those models have proven an invaluable tool for making statistical claims about the existence of stable vacua in various theories of supergravity. We will revisit one of these applications in \prettyref{sec:application}.

For a full understanding of a random variable, its distribution is of course of utmost importance. Unfortunately, for the probability density function (pdf) of a linear combination of $\chi^2$ random variables, there is no known closed analytic expression yet.
The great number of related publications---see \prettyref{sec:existing}---proves the ongoing interest in this topic.

It is surprising, however, that \emph{none} of the known methods for calculating this pdf seems computationally viable for more than a handful of terms, at most. In fact, the ones that are accurate enough suffer from exponential growth in the number of terms to calculate, see for example \cite{MATHAI_DAM} and the comment on this method in \cite{AKKOUCHI_MULTIINT_GAMMA}. Other known methods are discussed in \prettyref{sec:existing}.

Unfortunately, this lack of methodology proved to be an obstacle to making quantitative analytic claims about the existence of aforementioned vacua in a model of $4d$ random supergravity, see \cite{MARSH_RANDGRAV}. Motivated by this fact, we present a relatively simple procedure that is nonetheless far superior in speed and accuracy to the hitherto published methods.

\subsection{The Existing Literature}\label{sec:existing}
As already mentioned, there exist many approaches to calculating or approximating the distribution of a linear combination of $\chi^2$ random variables. For our purposes, one exclusion criterion is that the method has to be quantitatively accurate, which means that we need an explicit expression for the error. 
This excludes fits to other distributions, as for examples the three moment fit suggested by \cite{SOLSTEPH_CHI}, or the more recent mixture approximation by \cite{LINDSAY_CHI}.

Unfortunately, it also excludes saddlepoint or method of deepest descent approximations, see for example \cite[Sec. 5]{WoodSaddlepoint}, which involves several steps that make it hard to track the error accurately.
Let us discuss the techniques which pass this requirement.

One brute-force method is of course to draw a couple of million random variables, multiply them with the weights and calculate the distribution empirically. This works well when one is interested in the quantiles that lie around the bulk of the mass. In our example, however, the very low quantiles are of interest, and we simply do not have enough statistics to calculate those. Note that importance sampling can---unfortunately---not be used, since we do not know the complete distribution.

Another often-used ansatz is to expand the moment generating function in a series representation and inverting term by term. One such example (for gamma random variables) is given in \cite{MOSCH_SERIES_GAMMA}. Unfortunately, the prefactors in the expansion cannot be calculated in polynomial time.
%\footnote{$\propto 2^n$ in the definition of $\delta_k$}
Another example (for chi squares) is \cite{MOSCH_SERIES_CHI} where, again, the prefactors are hard to compute.
The method targeted at exponential distributions suggested in \cite{AKKOUCHI_MULTIINT_EXP} suffers from the same exponential growth, which is easy to see by looking at the suggested functional form.
%\footnote{$a_i$s grow exponentially}
A more recent example is \cite{MARTINEZ_SERIES}, which uses Laguerre polynomials. The same problem persists here; in addition, the error bound given suggests that obtaining an accurate result will require a very large number of terms.

As a third group, there are techniques that use numerical integration to obtain the density function, see for example \cite{FLEISS_NUMINT} and \cite{AKKOUCHI_MULTIINT_GAMMA}. Unfortunately, in both papers, the number of dimensions over which we have to integrate for $n$ terms is $n-1$. While this might be feasible for small $n$, it is not for our application---just sampling two points in any direction grows as $2^n$, and even using sophisticated \knownth{Monte Carlo} or \knownth{Gibbs sampling} algorithms does not help when $n$ is of order $1000$. Another method, which has only one integral to calculate, \cite{ROBERTS_INT}, has the problem of a highly oscillatory integrand that has to be integrated over $\sset R$, which seems impossible to do to the required accuracy.

To summarize, none of the above-mentioned techniques is feasible in our regime of interest---either due to computational restrictions or because the answer is not precise enough.
It should be noted, however, that calculating exponentially small derivations is not something one normally encounters in statistics, so it comes as no surprise that methods for that regime are yet to be developed.
In this paper, we develop a method which is both fast, arbitrarily accurate and computationally efficient even for a vast number of terms, especially in the low quantile regime that we are interested in.

%todo:Compare to exact values of exponential distributions in \cite{AKKOUCHI_MULTIINT_EXP}

\section{Analytic Results}\label{sec:method}
\subsection{Main Result}
We will first state the main result and prove the details in due course.

As stated in the first chapter, the problem is to calculate the density function of $Z=\sum_{i=1}^n a_iX_i$, where $X_i\sim\chi^2_r$ are iid random variables for some $r>0$ and $a_i\in\sset R_{>0}$.
Let $\op T_n:\C^n\rightarrow\sset R[n]$ denote the \knownth{Taylor-Maclaurin} map which assigns to a function its \knownth{Taylor} polynomial to order $n$ around $0$, and $\op R_n:\C^n\rightarrow\C^n$ its associated remainder map.

\begin{theorem}\label{th:lincomb}
Let $r\in2\sset N_0+1$, $Z=\sum_{i=1}^n\left(a_iX_i+b_iY_i\right)$, $X_i,Y_j\sim\chi^2_r$ iid random variables, $a_i,b_i\in\sset R_{>0}$ for some $n>0$ and $a_1<b_1<...<a_n<b_n$. Let $m_1,...,m_n\in\sset N_0$ and $\bar m:=\max\{m_1,...,m_n\}$. Then the density function $f_Z$ is given by
\[
f_Z(x)=C\theta(x)\left(\map T(x)+\map R(x)\right)\comma
\]
where $\bigsign[1.2]*_{i=1}^{n} f_i:=f_1\bar*...\bar*f_n$,
\[\map T(x)=\left(\bigast{i=1}{n}\ee^{-\frac{a_i+b_i}{4a_ib_i}\cdot}\left(\op T_{m_i}\map I_{\frac r2-\frac12}\right)\left(\frac{b_i-a_i}{4a_ib_i}\cdot\right)\right)(x)
\]
and
\[C=\left(\frac{\Gamma\left(\frac12-\frac r2\right)}{\Gamma(r)}\right)^{\!\!\!n}\prod_{i=1}^n(4a_ib_i)^{-r}\left(\frac{b_i-a_i}{8a_ib_i}\right)^{\!\frac12-\frac r2}\]
and 
\begin{align*}
0<\map R(x)<&\ \bar{\map R}(x)-\map T(X)\\
&:=\left(\bigast{i=1}{n}\ee^{-\frac{a_i+b_i}{4a_ib_i}\cdot}\left(\op T_{m_i}\map I_{\frac r2-\frac12}+\map R_i\right)\left(\frac{b_i-a_i}{4a_ib_i}\cdot\right)\right)(x)-\map T(x)\point
\end{align*}
Finally, $\forall x\le y\in\sset R_{>0}$, $\map R_i$ is bounded by
\begin{align*}
\map R_i(x)&:=\op R_{m_i}\map I_{\frac r2-\frac12}\left(\frac{b_i-a_i}{4a_ib_i}x\right)\\
&\le\frac{x^{m_i+1}}{(m_i+1)!}\ \left(\frac{b_i-a_i}{4a_ib_i}\right)^{m_i+1}\map I_{\frac r2-\frac12}^{(m_i+1)}\left(\frac{b_i-a_i}{4a_ib_i}y\right)\point
\end{align*}
\end{theorem}

This result---simple as it is---does not look particularly promising, since it contains a convolution of exponentially many terms. Fortunately, though, the next statement shows that this method is indeed surprisingly easy to compute, at least for $r=1$.
\begin{theorem}\label{th:lincombeff}
Let the setup be as in \prettyref{th:lincomb}.
 Then the computational complexity to calculate a value $x$ of $f_Z$ with relative error smaller than $R_\mathrm{max}\in(0,1)$ 
is $\map O(n^2r^2 \bar y^2 (\log rn-\log(R_\mathrm{max})))$, where $\bar y:=x\max_{i=1,...,n}\{\frac{b_i-a_i}{4b_ia_i}\}$.
\end{theorem}

In other words, this means that the algorithm is at most polynomial in the number of terms in our sum, polynomial in the degrees of freedom, polynomial in $\bar y$ as given above and logarithmically in the precision we want to obtain.

Our numerical analysis shows that this bound is by no means optimal.

We will now give a rather detailed proof. The mathematics involved is not particularly deep, but contains some ideas that are essential for later use.

\subsection{The Linear Combination of Two Gamma and $\chi^2$ Random Variables}
\begin{definition}Let $X\sim\Gamma(\alpha,\beta)$ be a gamma-distributed random variable with shape parameter $\alpha\in\sset R_{>0}$ and rate parameter $\beta\in\sset R_{>0}$. Then its density function $f_X$ is given by
\begin{equation}
f_X(x)=\theta(x)\frac{\beta^\alpha}{\Gamma(\alpha)} x^{\alpha-1} e^{-\beta x}\point
\end{equation}
\end{definition}
Note the following
\begin{remark}
If $X\sim\Gamma(\alpha,\beta)$, then $cX\sim\Gamma(\alpha,c\beta)$ for $c\in\sset R_{>0}$.
\end{remark}
This is why $\beta$ is also called \emph{inverse shape parameter}.
This immediately leads me to the following
\begin{lemma}\label{lem:gammaconv}
Let $Z=X+Y$ where $X\sim\Gamma(\alpha_1,\beta_1)$ and $Y\sim\Gamma(\alpha_2,\beta_2)$ are independent Gamma distributions. Then the density function $f_Y$ is given by
\begin{equation}\label{eq:gammaconv}
f_Z(z)=\theta(z)\frac{\beta_1^{\alpha_1}\beta_2^{\alpha_2}}{\Gamma(\alpha_1+\alpha_2)} z^{\alpha_1+\alpha_2-1} \ee^{-\beta_2 z} {_1\map F_1}(\alpha_1;\alpha_1+\alpha_2;(\beta_2-\beta_1)z)\comma
\end{equation}
where $_1\map F_1$ is a confluent hypergeometric function (Kummers function of the first kind).
\end{lemma}
\begin{proof}
We have to calculate the convolution $f_Z\equiv f_X*f_Y$ over $\sset R$.
\begin{align*}
(f_X&*f_Y)(z)\\
&=\frac{\beta_1^{\alpha_1}\beta_2^{\alpha_2}}{\Gamma(\alpha_1)\Gamma(\alpha_2)}\int_{\sset R}\theta(x)\theta(z-x) x^{\alpha_1-1}(z-x)^{\alpha_2-1}\ee^{-\beta_1 x-\beta_2(z-x)}\dd x\\
&=\frac{\beta_1^{\alpha_1}\beta_2^{\alpha_2}}{\Gamma(\alpha_1)\Gamma(\alpha_2)} \ee^{-\beta_2 z} \theta(z) \int_0^z x^{\alpha_1-1}(z-x)^{\alpha_2-1}\ee^{(\beta_2-\beta_1) x}\dd x\comma
\end{align*}
which can be further simplified by substituting $x=zt\Rightarrow\dd x=z\dd t$, so
\begin{align*}
(f&_X*f_Y)(z)\\
&=\theta(z) \frac{\beta_1^{\alpha_1}\beta_2^{\alpha_2}}{\Gamma(\alpha_1)\Gamma(\alpha_2)} \ee^{-\beta_2 z} z^{\alpha_1-1+\alpha_2-1+1}\int_0^1t^{\alpha_1-1}(1-t)^{\alpha_2-1} \ee^{(\beta_2-\beta_1) zt}\dd t\\
&\equiv\theta(z)\frac{\beta_1^{\alpha_1}\beta_2^{\alpha_2}}{\Gamma(\alpha_1)\Gamma(\alpha_2)} \ee^{-\beta_2 z} z^{\alpha_1+\alpha_2-1}\frac{\Gamma(\alpha_1)\Gamma(\alpha_2)}{\Gamma(\alpha_1+\alpha_2)}{_1\map F_1}(\alpha_1;\alpha_1+\alpha_2;(\beta_2-\beta_1)z)\\
&=\theta(z)\frac{\beta_1^{\alpha_1}\beta_2^{\alpha_2}}{\Gamma(\alpha_1+\alpha_2)} z^{\alpha_1+\alpha_2-1} \ee^{-\beta_2 z} {_1\map F_1}(\alpha_1;\alpha_1+\alpha_2;(\beta_2-\beta_1)z)\comma
\end{align*}
where we have used the integral representation of $_1\map F_1$ in the second to last line.
\end{proof}
Observe that \prettyref{eq:gammaconv} is indeed symmetric under exchange of the subscripts ${}_1\leftrightarrow{}_2$, as expected.

\begin{definition}\label{def:chisq}
Let $X\sim\chi^2_k$ be a chi-square random variable with $k$ degrees of freedom. Then the density function $f_X$ is given by
\begin{equation}
f_X(x)=\theta(x)\frac{x^{\frac{k}2-1} e^{-\frac x2}}{2^{\frac k2}\Gamma\left(\frac k2\right)}\ \point
\end{equation}
\end{definition}

\begin{remark}\label{rem:chigamma}
If $X\sim\chi^2_k$, then $aX\sim\Gamma\left(\frac k2, \frac1{2a}\right)$ for $a\in\sset R_{>0}$.
\end{remark}

As a second important result, which will find its application later, we conclude the following
\begin{corollary}\label{cor:chiconv}
Let $X,Y\sim\chi^2_k$ two iid chi-square random variables with $k$ degrees of freedom. Let $Z:=aX+bY$, $a,b\in\sset R_{>0}$. Then the density function $f_Z$ is given by
\begin{equation}\label{eq:chiconv}
f_Z(z)=\theta(z)\frac1{(4ab)^k} \left(\frac{a-b}{8ab}\right)^{\frac12-\frac k2}\frac{\Gamma\left(\frac12+\frac k2\right)}{\Gamma(k)}\ \ee^{-\frac{a+b}{4ab}z}\ x^{\frac k2-\frac12}\ \map I_{\frac k2-\frac 12}\left(\frac{b-a}{4ab}z\right)\comma
\end{equation}
where $\map I_\nu$ is the $\nu$th order modified \knownth{Bessel} function of the first kind.
\end{corollary}
\begin{proof}
This follows from \prettyref{lem:gammaconv}, \prettyref{rem:chigamma} and from the identity (known as \knownth{Kummer's second transform})
\[{_1\map F_1}(\alpha;2\alpha;x)\equiv \ee^\frac x2\left(\frac x4\right)^{\frac 12-\alpha}\Gamma\left(\frac 12+\alpha\right)\map I_{\alpha-\frac12}\left(\frac x2\right)\point\qedhere\]
\end{proof}

Equation \ref{eq:chiconv} struck my interest for the following reason: for small $z$, given $a$ and $b$ are of the same order, the prefactor in the modified \knownth{Bessel} function is small, so a promising ansatz for an approximation should be a \knownth{Taylor-Maclaurin} series for $\map I_\nu$. The exponential then restricts most of the mass to lower values of $z$, which, in addition, suppresses the error introduced from the expansion. Furthermore, the leading nontrivial expansion term of $\map I_0(z)$ is of order $\set O(z^2)$ and for general $\map I_\nu$ only every second order monomial is present. All this will be quantified in the next chapters.

At first, though, we need some more results that will come into play later.

\subsection{Convolution Algebra}
\begin{definition}
Let $\bar*:\C(\sset R)\times\C(\sset R)\rightarrow\C(\sset R)$ be defined as
\[(f\ \bar*\ g)(z)=\int_{\sset R}\theta(x)\theta(z-x)f(x)g(z-x)\dd x\equiv\int_0^z f(x)g(z-x)\dd x\]
for $f,g\in\C(\sset R)$.
\end{definition}
This map is well-defined, as follows from \cite{STEINWEISS_FOURIER}, Th. 1.3, and could well be extended to spaces such as $\set L^1(\sset R)$, but we do not need more here. Important is the following
\begin{lemma}\label{lem:convalg}
Let $\set P_N:=\{\ee^{b_ix}x^{n_i}:b_i\in\sset R,n_i\in\sset N_0,n_i<N\}$ where $N\in\sset N$. Let further $f,g\in\set P_N$: $f(x)=\ee^{ax}x^n$, $g(x)=\ee^{bx}x^m$ and define
\[(f\ \tilde*\ g)(x):=\left\{
  \begin{array}{lr}
    0 &: a=b\\
    (f\ \bar*\ g)(x) &: a\neq b
  \end{array}
\right.\point\]
Then $(\left<\set P_N\right>_{\sset R},\tilde*)$ forms an algebra over $\sset R$ $\forall N\in\sset N$, where $\tilde*$ extends canonically to the linear hull.
\end{lemma}
\begin{proof}
First note that the use of $\tilde*$ over $\bar*$ is just a technicality---we will not encounter the case of two equal exponentials in our later application. We see that we can use $\bar*$ as long as our exponentials are distinct.

Including the case $a=b$---in which case $f\bar*g$ is clearly nonzero---introduces higher powers of the monomials $x^n$, which would pose a problem later on.

We will only show that the algebra is closed, the other properties are easily derived.
Analogously to the proof of \prettyref{lem:gammaconv}, we have
\[(f\ \tilde*\ g)(x)=\theta(x)\ee^{bx} \frac{\Gamma(m+1)\Gamma(n+1)}{\Gamma(n+m+2)} x^{m+n+1} {_1\map F_1}(m+1;n+m+2;(a-b)x)\point\]
Using that, for integers $r,s\in\sset N_{>0}$, $r<s$, we have (see \cite{W_CONFLUENT})
\begin{align*}
{_1\map F_1}(r;s;z)\equiv&\ \frac{(s-2)!(1-s)_r}{(r-1)!}\ z^{1-s}\ \cdot\\
&\left( \sum_{k=0}^{s-r-1} \frac{z^k (-s+r+1)_k}{k! (2-s)_k}-\ee^z \sum_{k=0}^{r-1}\frac{(-z)^k(1-r)_k}{k!(2-s)_k}\right)\comma
\end{align*}
we note that the lowest occurring power of $z$ is of order $\set O(z^{1-s})$ or, including the previous result, a constant. The highest occuring power is of order $\set O(z^{-r})$ resp. $\set O(z^n)$ or $\set O(z^{r-s})$ resp. $\set O(z^m)$.
\end{proof}

Lemma \ref{lem:convalg} is an important result, as it claims that convolving two terms does not result in higher order monomials in our expression. But we can make even stronger claims.
\begin{remark}\label{rem:conv2}
Let $f_1,...f_n\in\set P_N$, $g_1,...,g_m\in\set P_N$, all $f_i,g_j$ having pairwise distinct exponentials, and $a_1,...,a_n,b_1,...,b_m\in\sset R$. Then $\exists l\le (n+m)N$, $h_1,...,h_l\in\set P_N$ pairwise distinct and $c_1,...,c_l\in\sset R$:
\[\left(\sum_{i=1}^na_if_i\ \bar*\ \sum_{j=1}^mb_jg_j\right)(x)=\sum_{i=1}^lc_ih_i(x)\point\]
\end{remark}
\begin{proof}
Here, the same notation as in the proof for \prettyref{lem:convalg} is used. Counting the number of independent exponentials $l_e$, we obtain $l_\ee\equiv n+m$. The maximum number of monomials multiplied with an exponential is then $\max\{s-r-1,r-1\}=\max\{n,m\}\le N$.
\end{proof}
This by itself can still grow very quickly for more than two convolutions, but the next result puts a stricter bound on the number of terms if we convolve multiple times.
\begin{lemma}\label{lem:conveff}
Let $f_1,...,f_q\in\langle\set P_N\rangle_{\sset R}$ so that $f_i$ has $n_i$ distinct exponentials, and $a_1,...,a_q\in\sset R$. Then $\exists l\le N\sum_{i=1}^qn_i$, $h_1,...,h_l\in\set P_N$ pairwise independent and $c_1,...,c_l\in\sset R$:
\[(a_1f_1\ \bar*\ ...\ \bar*\ a_qf_q)(x)=\sum_{i=1}^lc_ih_i(x)\point\]
\end{lemma}
\begin{proof}
Let us consider the first convolution. We first note that the number of distinct exponentials $l_e$ in \prettyref{rem:conv2} is $l_\ee=n_1+n_2$, even if we replace the monomials with arbitrary polynomials. The rest then follows from induction and the fact that $(\left<\set P_N\right>_{\sset R},\tilde*)$---being an algebra---is closed.
\end{proof}
With these basic facts, let us now discuss the details for \prettyref{th:lincomb} and \prettyref{lem:conveff}. actual algorithm to calculate the distribution of a linear combination of chi square random variables.

\section{The Density Function of a Linear Combination of $\mathbb\chi^2$ Random Variables}
\subsection{Density Function}

% main theorem

\begin{proof}[Proof of \prettyref{th:lincomb}]
The main statements all follow from \prettyref{cor:chiconv} and are a straightforward application of \knownth{Taylor's} theorem, where $\map I_{\frac r2-\frac12}\in\C^\infty(\sset R)$ $\forall r\in2\sset N_0+1$. The estimate for the remainder is a simple uniform estimate, using the fact that $\map I^{(k)}(x)$ is strictly monotonically increasing $\forall x\in\sset R_{>0},\nu\ge0$ and $k\in\sset N_0$.
\end{proof}
It might seem like a restriction that $r$ has to be odd. Note though that the case $r$ even is simple, since---as already mentioned---we would only have to convolve terms from $\set P_N$ (see \prettyref{def:chisq}). Also observe that it is not of great significance that this expansion holds only for an even number of terms, as convolving numerically \emph{once} is easily accomplished.

The claim of \prettyref{th:lincombeff} is that this density is very efficient to calculate to high precision, especially if the argument is not too big. To quantify this a bit more, we have to relate the magnitude of the error to the prefactors $a_i,b_i$ and the abscissa $x$.

\subsection{Error Estimate and Complexity Class}
Let us consider the general case. We are interested in all $x\in[0,x_\mathrm{max})=:\Omega$ for some $x_\mathrm{max}>0$. Let $f,g\in\set L^1(\Omega)\cap\C^\infty(\Omega)$ positive and $\op T_ng\uparrow g$. Then
\begin{align*}
|(f\ \bar*\ &(g-\op T_n g))(x)|=\int_\Omega f(x-y)(g-\op T_n g)(y)\dd y\\
&\le\int_\Omega\|f\|_\infty^\Omega(g-\op T_n g)(y)\dd y=x_\mathrm{max}\|f\|_\infty^\Omega\int_\Omega(g-\op T_n g)(y)\dd y\point
\end{align*}
This of course also holds when we have an additional weight function $\eta$ in the product, which we will denote by $\bar*_\eta$. Assume now we want to have the error ratio below some fixed $R_\mathrm{max}>0$. Then one can estimate
\begin{align*}
\frac{(f\ \bar*_\eta\ (g-\op T_n g))(x)}{(f\ \bar*_\eta\ g)(x)}&\le\frac{x_\mathrm{max}\|\eta f\|_\infty^\Omega\int_\Omega\eta(y)(g-\op T_n g)(y)\dd y}{x_\mathrm{max}\|\eta f\|_\infty^\Omega\int_\Omega\eta(y)g(y)\dd y}\\
&=1-\frac{\int_\Omega\eta(y)(\op T_n g)(y)\dd y}{\int_\Omega\eta(y)g(y)\dd y}=1-\frac{\|\eta(\op T_n g)\|_1^\Omega}{\|\eta g\|_1^\Omega}\overset!\le R_\mathrm{max}
\end{align*}
or
\begin{equation}\label{eq:errorbound1}
\|\eta(\op T_n g)\|_1^\Omega\overset!\ge(1-R_\mathrm{max})\|\eta g\|_1^\Omega\point
\end{equation}
For the numerical algorithm, this is basically what one uses to determine how many \knownth{Taylor} terms have to be kept to satisfy the inequality. This integration can be performed very fast numerically, since the integrand is sufficiently regular and the integration region compact.

Since this result does not depend on $f$---which is one of the reasons why the error is exaggerated gravely, more of this later---we can inductively determine how many terms we have to keep: in the most na\"ive way possible, for $m$ functions, we remain below the error bound $R_\mathrm{max}$ if each convolution separately satisfies the error bound $\frac{R_\mathrm{max}}{m-1}$.

Just to get the order of the complexity class right, let us consider our special case
\[\eta(x)=\ee^{-\frac{a+b}{4ab}x}\mt{and}g(x)=\map I_{\frac r2-\frac12}\left(\frac{b-a}{4ab}x\right)\mt{for}a,b\in\sset R_{>0}\point\]
Then
\[\|\eta(\op T_n g)\|_1^\Omega\ge\|\eta\|_1^\Omega\|T_ng\|_\infty^\Omega\]
using \knownth{Hölder's inequality}, and applying the same to the right hand side of \prettyref{eq:errorbound1}, we obtain the relation
\[\|\op T_n g\|_\infty^\Omega\overset!\ge(1-R_\mathrm{max})\|g\|_\infty^\Omega\point\]
Since $\map I_\nu$ is monotoneously increasing, this is equivalent to
\begin{align*}
(\op T_n g)(x_\mathrm{max})&\ge(1-R_\mathrm{max})g(x_\mathrm{max})
\shortintertext{or}
(\op R_n g)(x_\mathrm{max})&\le R_\mathrm{max}g(x_\mathrm{max})\point
\end{align*}
Let now $y:=\frac{b_i-a_i}{4a_ib_i}x_\mathrm{max}$. Using the uniform estimate from \prettyref{th:lincomb},
\[\frac{y^{n+1}}{(n+1)!}\ \map I_{\frac r2-\frac12}^{(n+1)}(y)\overset!\le R_\mathrm{max}\map I_{\frac r2-\frac12}(y)\point
\]

%todo: simulate that n-dependence for various r's
For $r=1$, e.g. if all our random variables are $\chi^2_1$-distributed, we know that $\map I_0^{(n)}(x)<\map I_0(x)\ \forall x\ge0$. We need
\[\frac{x^n}{n!}\le\ee^{-r}\quad\Leftarrow\quad\log\left(\frac {x^2}n\right)\le-\frac rn\quad\Leftarrow\quad n\ge r\map W^{-1}\left(\frac r{x^2}\right)\comma\]
where, in the second step, the inequality $n!\ge n^{\frac n2}$ is used. For large $n$, it could be replaced by $n!\ge n^{rn}$ for some $r=r_n\in\left[\frac12,1\right)$. $\map W$ is the \knownth{Lambert} $\map W$ function. With this result, we finally require
\begin{align*}
\frac{y^{n+1}}{(n+1)!}\overset!\le R_\mathrm{max}=:\ee^{-r_\mathrm{max}}
\mt{or\ roughly}
n\in\set O(r_\mathrm{max}y_\mathrm{max}^2)\point
\end{align*}
Note that these final estimates tremendously overestimate the error and should not be used to actually determine the number of terms that have to be kept. Equation \ref{eq:errorbound1} gives a far more reliable result.

Combining this with our results from examining the convolution algebra, \prettyref{lem:conveff}, we can finally prove our second main result.

\begin{proof}[Proof of \prettyref{th:lincombeff}]
Consider the $i$th convolution. A straightforward algorithm has to convolve $N_1\cdot i$ terms with $N_2\cdot 1$ terms. Using the notation from above, an upper bound surely is $N_1=N_2\in\set O\left(\left|\log(\frac{R_\mathrm{max}}{n-1})\right|\bar y^2\right)\equiv\set O((\log(n)-\log(R_\mathrm{max}))\bar y)$, thus---in a straightforward implementation---we will have to convolve the terms pairwise, i.e. $i\left((\log(n)-\log(R_\mathrm{max}))\bar y^2\right)^2$ times. Summing all steps up, one concludes that the number of terms that have to be convolved are
\begin{align*}
\map O\!&\left(\sum_{i=1}^{n-1}i\left((\log(n)-\log(R_\mathrm{max}))\bar y^2\right)^2\right)\\
&\qquad=\map O\!\left((\log(n)-\log(R_\mathrm{max}))\bar y^4\cdot\frac12n(n-1)\right)\\
&=\map O\!\left(\frac{n^2\log(n)}{\log(R_\mathrm{max})} \bar y^4 \right)\point
\end{align*}

We can recover the case of arbitrary degrees of freedom $r$ by just adding $r$ $\chi^2_1$ random variables, so scaling $n$ by $r$ yields the desired result.
\end{proof}

This of course neglects implementation details---we assume for example that the integration estimate \prettyref{eq:errorbound1} is constant for all terms. The convolution itself can be done by multiplying matrices and then sorting vectors, which is more likely to be in $\set O(n\log n)$, but this would only introduce a minor change in our end result.

Moreover, one finds that for our application (see \prettyref{sec:application}), the overall growth is rather linear in $n$, so it remains an open question to obtain a better estimate on the complexity of the algorithm and to generalize the result to arbitrary $r$.

\subsection{A Better Error Estimate}
Just as a side note and to prove that it is rather simple to optimize this algorithm further, note the following
\begin{remark}
Let $\Omega\subset\sset R$ be a closed interval, $f\in\C^\infty(\Omega):f^{(k)}(x)\ge0\ \forall x\in\Omega$ and $m\ge n\in\sset N_0$ arbitrary. Then
\[(\op R_nf)(x)\le\left((\op T_m-\op T_n)f\right)(x)+M_m\quad\forall x\in\Omega\comma\]
where $(\op R_mf)(x)\le M_m\ \forall x\in\Omega$.
\end{remark}
This obvious application of the $Lagrange$ remainder formula can be used to reduce the expansion order for the error bound. We thus save ourselves almost \emph{half} of the work, namely the expansion of the estimate polynomial. We have found this improvement to be significant.

%todo: graph?

\section{Numerical Analysis}
We now turn to a numerical analysis of our proposed technique, \prettyref{th:lincomb}. The implementation is quite straightforward and done using \textsc{Mathematica}. As already mentioned, \prettyref{eq:errorbound1} is used to decide how many \knownth{Taylor} terms to keep in each step, i.e. the $m_i$ are determined this way. The actual convolution can be implemented very efficiently as a simple algebraic method, since we never leave our $\bar*$ algebra.
%\footnote{\textsc{Mathematica}'s implementation of \texttt{Convolve} is way too slow.}

\subsection{Comparisons with other Methods}

\begin{figure}
\subfigure[Density, few expansion terms]{\includegraphics[width=0.49\textwidth]{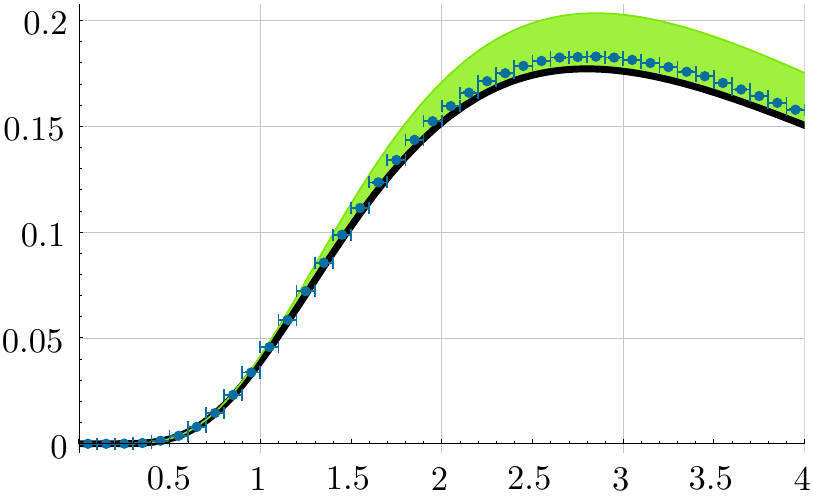}}\hfill
\subfigure[Density, many expansion terms]{\includegraphics[width=0.49\textwidth]{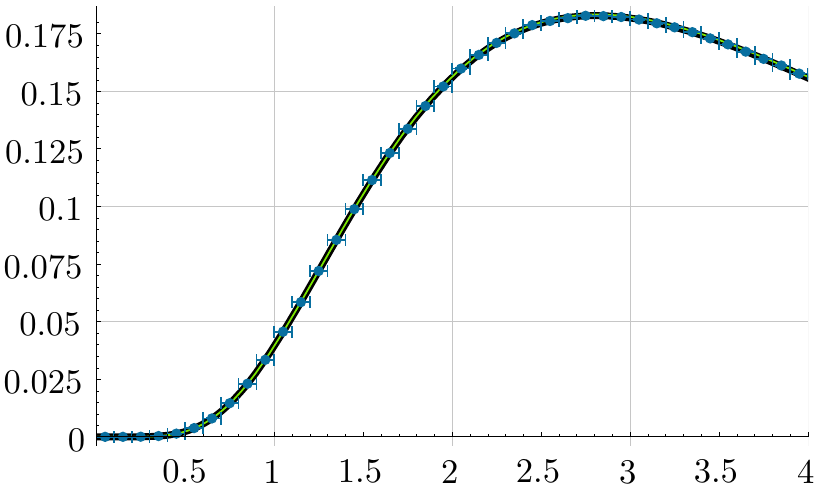}}
\subfigure[Absolute errors, few expansion terms]{\includegraphics[width=0.49\textwidth]{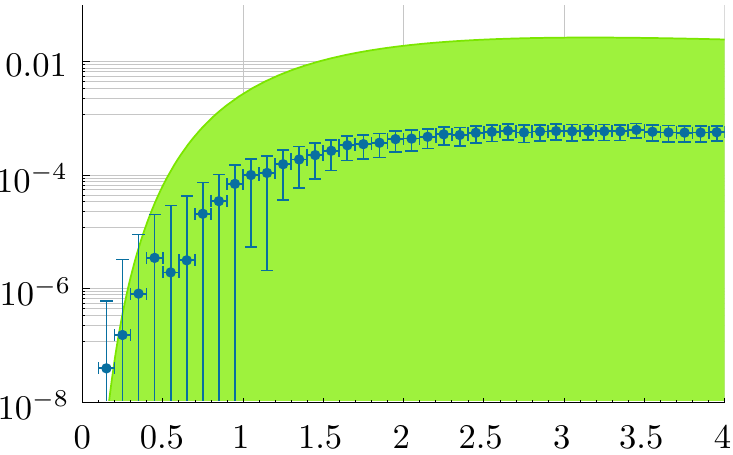}}\hfill
\subfigure[Absolute errors, many expansion terms]{\includegraphics[width=0.49\textwidth]{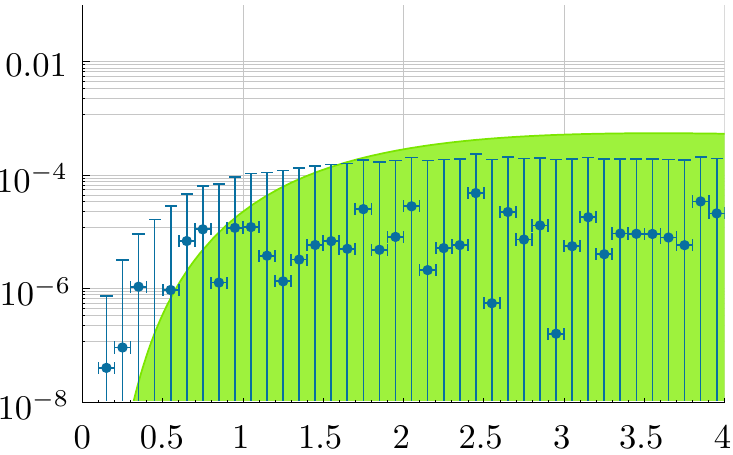}}
\caption{The distribution for $20$ $\chi^2$ random variables. Figure (a) shows the density compared to an empirical histogram distribution. Figure (c) shows the absolute predicted error vs. the true error, compared to the histogram distribution. (b) and (d) is are the same, with more expansion terms. Note that the uncertainty in the histogram distribution is---due to computational restrictions---huge.}
\label{fig:distcomp1}
\end{figure}
One way of checking our results quantitatively is comparing them to \knownth{Monte Carlo} simulations, namely building a histogram distribution from a table of drawn values. This will only work in the regime with a lot of mass in the distributions, the tails will not receive enough data points to be quantitatively correct.

In \prettyref{fig:distcomp1}, you can see the density for $n=20$, where the weights were taken from \prettyref{eq:marpast}. Depending on the threshold that is taken for the error bound, \prettyref{eq:errorbound1}, the expansion comes arbitrarily close to the empirical distribution. Note that the error bound shrinks accordingly and---as expected---we only have an error in the positive $y$ direction. Also observe that our error bound still vastly overestimates the true error.

\begin{figure}
\subfigure[\knownth{Laguerre} expansion]
% SUBSTITUTE this graphic with 6coplag_bw for print
{\includegraphics[width=0.49\textwidth]{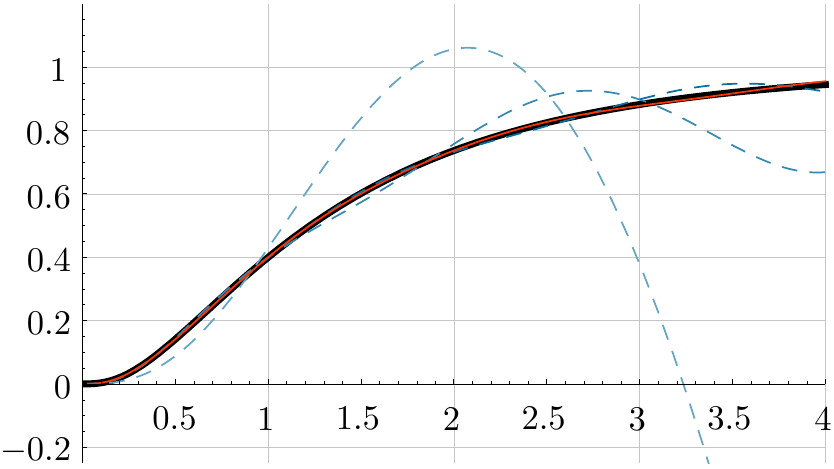}}\hfill
\subfigure[Numerical integration]{\includegraphics[width=0.49\textwidth]{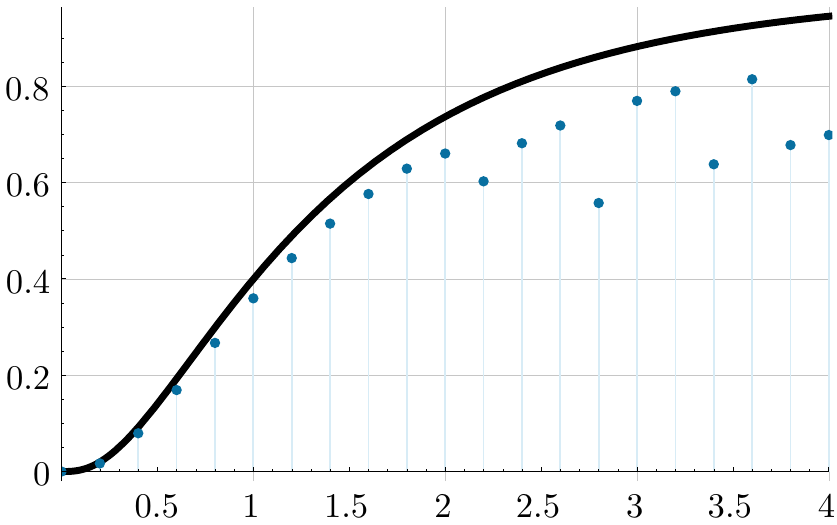}}
\caption{The cumulative distribution for $6$ $\chi^2$ random variables. Figure (a) shows an expansion in Laguerre polynomials as proposed by \cite{MARTINEZ_SERIES}, with $1$, $6$, $11$ (blue, dashed) and $16$ (orange, overlays thick black line) expansion terms. Figure (b) shows a numerical integration proposed by \cite{AKKOUCHI_MULTIINT_GAMMA}. The black line is my method, which is exact at the shown resolution.}
\label{fig:comp1}
\end{figure}
In \prettyref{fig:comp1}, my method is compared to two more recently published techniques, where the six weights are again taken from \prettyref{eq:marpast}. The series expansion (a) can be made quite accurate (although the error bound---not shown---is very bad and suggests we expand to much higher order), at the cost of exponentially growing computation time. For more than $10$ terms, it seems impossible to expand to high enough order in any sensible time.

On the other hand, the numerical integration---performed with \textsc{Mathematica}'s adaptive quasi \knownth{Monte Carlo} method with about $10^5$ integrand evaluations---suffers from severe convergence problems.

Just as a comparison: The \knownth{Laguerre} expansion takes about $10$ seconds for $16$ terms, the numerical integration about $250$ seconds (for all the values shown---for a single value it is about $10$ seconds), whereas my method needs about $0.2$ seconds. While not representative, this should give a rough idea of what to expect; they were all run on one i7-2670QM core. While it is true that implementing the \knownth{Laguerre} expansion or the numerical integration in C could be orders of magnitudes faster, mind that our method is executed by \textsc{Mathematica}, too.

\section{Counting String Vacua}\label{sec:application}
\subsection{Background and Motivation}
String theory admits some $10^{500}$ vacuum solutions---for an interesting digression, see \cite{SUSSKIND_STRINGVAC}. They emerge from the parameter space describing the internal compact manifold, which string theory requires in addition to our $3+1$ dimensional spacetime. This parameter space is called the \emph{moduli space of supersymmetric vacua}. At low energies, these moduli appear as massles scalar fields with their own equations of motion and potential.

The question of how many of those configurations correspond to physically stable solutions is an active field of research.
One vacuum configuration---a \knownth{deSitter} spacetime---is particularly physically motivated, since it agrees with the recent discovery of an accelerated expanding universe.

A more recently published work by \cite{Bachlechner:2012at} analytically calculates the probability of metastable vacua in exactly supersymmetric scenarios.

In a general supersymmetric AdS vacuum, all the masses are bigger than the Breitenlohner-Freedman bound, which is $m_{min}^s=-\frac94|W|^2<0$ in four dimensions (see $2.7$ in \cite{Bachlechner:2012at}). Thus the probability that all masses are positive, has been calculated to go as 
\[P_N=\exp(-c_2 N^2)\quad\text{where}\quad c_2=2\frac{|W|^2}{m_\mathrm{susy}^2}\comma\]
where $W$ denotes the superpotential.
By tuning this parameter $W/m_\mathrm{susy}$ accordingly, one can make the probability of tachyonic directions either arbitrarily small or large---the interesting question is what a typical value of this parameter is. Some answers to this question can also be found in \cite{Bachlechner:2012at}.

In \cite{MARSH_RANDGRAV}, $W$ and $F_I={\cal D}_IW$ are very small, approaching a supersymmetric \emph{Minkowski} vacuum, where there is no supression---we have $W\equiv0$ and all masses are positive. But since the authors want to keep a deSitter vacuum---i.e. only approximately approach a Minkowski vacuum---the sGoldstino in general has a negative mass.
What was calculated in \cite{MARSH_RANDGRAV} was the probability that \emph{this} direction is not tachyonic. One important point to notice is that it is not at all obvious that \knownth{deSitter} vacua arising from spontaneous symmetry breaking are much more common than \knownth{deSitter} vacua from uplifting of a supersymmetric \knownth{AdS} vacuum. This question has to be addressed by calculation and is not something we know \knownth{a priori}.

Following the work of \cite{MARSH_RANDGRAV}, in the study of $\mathcal N=1$ supergravities, it was shown that the fraction of metastable \knownth{deSitter} vacua shrinks exponentially with the number of scalar fields $N\gg1$. By constructing a random matrix model for the Hessian matrix $\op H$ of the scalar potential and calculating its eigenvalue spectrum, local minima can be found.

The suggested functional dependence for $N$ fields for the probability of the smallest eigenvalue fluctuating to a positive value---which renders $\op H$ positive definite---was
\begin{equation}\label{eq:statdep}
P_N\propto\exp\left(-c N^p\right)\mt{where}c,t>0\point
\end{equation}

The Hessian matrix can be block-diagonalized to $\op H=\map{diag}(m_-,m_+)$, where $m_\pm$ denotes a mass matrix. An important result is the explicit expression for the smallest eigenvalue of the mass matrix,
\[m_\pm^2=F^2\mathcal T_\pm-F^2\mathcal S\point\]
$\mathcal T_\pm$ and $\mathcal S$ are given by
\begin{align}
\mathcal T_\pm&=\frac23\omega^2+K_{11}^eK_{\bar1\bar1e}-K_{1\bar11\bar1}\pm\underbrace{\left|U_{111}F^{-1}\ee^{-\theta_F}-\frac23\omega^2 \ee^{2\ii(\theta_F-\theta_W)}\right|}_{=:|t_\mathrm{hol}|}\comma\label{eq:mT}\\
\mathcal S&=\sum_{b'=2}^N\frac{|U_{11b'}|^2}{\lambda_{b'}^2}
\intertext{where}
\omega&=\frac{\sqrt3}F|W|\mt{and}F\propto\frac1{\sqrt N}\point\notag
\end{align}
Here, $K$ denotes the Kähler potential and $W$ the superpotential, which are taken to be random functions. $U$ and $F$ are derivatives of $W$.
It was claimed in \cite{MARSH_RANDGRAV} that
%\footnote{I could confirm this explicitly.}
the statistical properties of $\mathcal T_\pm$ can be neglected for large $N$ and replaced by its expectation value, i.e. $\mathcal T_-\approx1$. While qualitatively correct for the limiting distribution, this approximation neglects a number of important aspects, so let us revisit this argument  briefly.

The first term in $\mathcal T_-$, $\omega^2$, does not depend on $N$ and we can thus ignore it, since we rescale our results anyways---for $\omega\in[0,1]$, the introduced error is negligible. The term $-|t_\mathrm{hol}|$ is indeed negative semi-definite, and for a conservative estimate we can set $|t_\mathrm{hol}|$ to $0$, as expained in \cite{MARSH_RANDGRAV}.

For $|K^{(3)}|^2$, which is distributed as $\frac1N\chi^2_N$ with $\langle|K^{(3)}|^2\rangle=1$, it was claimed that the fluctuation probability is negligible for large $N$, stating the central limit theorem. While certainly correct for very large $N$, distributions with nonzero skewness tend to converge slowly to a normal distribution. Furthermore, the width of this process scales as $N^{-1/2}$.
The same argument holds true for the third term in \prettyref{eq:mT}, $K_{1\bar11\bar1}^{(4)}\sim\mathcal N(0,N^{-1/2})$, so even for large $N=\mathcal O(10^3)$, fluctuation of these two terms cannot be dismissed.

\ref{sec:afprob} contains an overview over the different terms comprising $m^2_-$ where the just-mentioned points can be seen explicitly. Comparing the convolutions from the middle and right column, the reader can observe that the fluctuation probability of the $\mathcal T$-term indeed plays an important role, and still affects the overall magnitude quite significantly, even for larger $N$.

For the sake of demonstrating the usefulness of our method to the present problem, however, we will revert temporarily to ignoring the $\mathcal T$-terms and try to calculate the leading contribution to the fluctuation probability of $m^2_-$,
i.e. $\p(\mathcal S\le1)$. We will reintroduce the $\mathcal T$ contributions in the next section.

The $U_{11b'}$ are $N-1$ independent and identically normally distributed random variables $\sim\mathcal N(0,1)$, while the denominators $\lambda_{b'}^2$ are determined by the \knownth{Marčenko-Pastur} law
\begin{equation}\label{eq:marpast}
f_\mathrm{MP}(\lambda)=\frac1{2\pi\lambda N\sigma^2}\sqrt{\lambda\left(4N\sigma^2-\lambda\right)}\mt{where} \sigma=\frac1{\sqrt N}\point
\end{equation}
It can be justified to take the weights $\lambda_{b'}^{-2}$ such that
\begin{equation}\label{eq:eigenvalues}
\lambda_{b'}^2=\avg{\lambda_{b'}^2} \mt: \int_0^{\avg{\lambda_{b'}^2}}f_\mathrm{MP}(\lambda)\dd\lambda=\frac{b'}N\point
\end{equation}
This corresponds to numerically inverting the function $f_\mathrm{MP}$ on the range $(0,4)$, which is well defined (unfortunately, this cannot be done analytically).
 $f_\mathrm{MP}$ and the weights for specific $N$ can be seen in \prettyref{fig:mpastur}.
\begin{figure}
\subfigure[]{\includegraphics[width=0.49\textwidth]{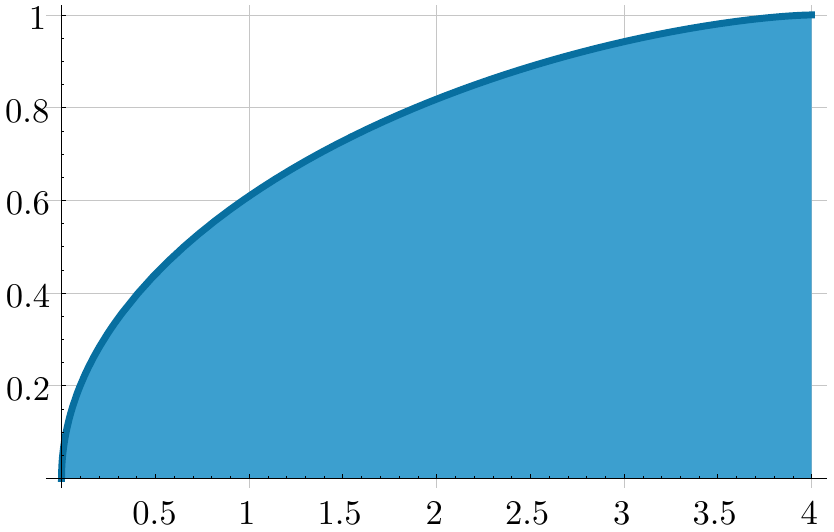}}\hfill
\subfigure[]{\includegraphics[width=0.49\textwidth]{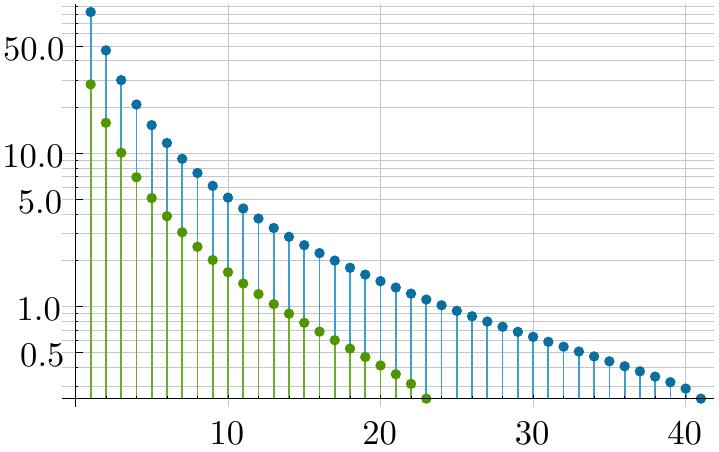}}
\caption{(a) The \knownth{Marčenko-Pastur} law. (b) Weights for $N=24$ and $N=42$, as determined by \prettyref{eq:eigenvalues}.}
\label{fig:mpastur}
\end{figure}

It is important to point out that the prefactors pose another hard condition on the problem. Rather elegant methods such as the $\chi^*$ approach in \cite{LindsayChiStar} cannot be used to extract further useful information about this specific case of a linear combination of $\chi^2$ random variables, since a closed expression for the prefactors is not known.

In order to calculate the fluctuation probability $\p(\mathcal S\le1)$, we need to know the distribution of $\mathcal S$, which is a linear combination of iid $\chi^2_1$ random variables. In \cite{MARSH_RANDGRAV}, a three moment fit (\cite{SOLSTEPH_CHI}) is used to approximate the distribution with a single $\chi^2_p$ variable---the quality of which we will discuss later.
The computed values were
\begin{align*}
p&\approx 0.24\mt{and}c\approx23\point
\intertext{In contrast, by simulating the full mass matrix $m^2_-$ numerically and then fitting the $N$-dependence to \prettyref{eq:statdep}, the values}
p&=1.28\pm0.10\mt{and}c=0.29\pm0.06
\end{align*}
were obtained. The difference is significant.

\subsection{Simulation}

\begin{figure}
\subfigure[$P_N$]{\label{fig:sim-a}\includegraphics[width=0.49\textwidth]{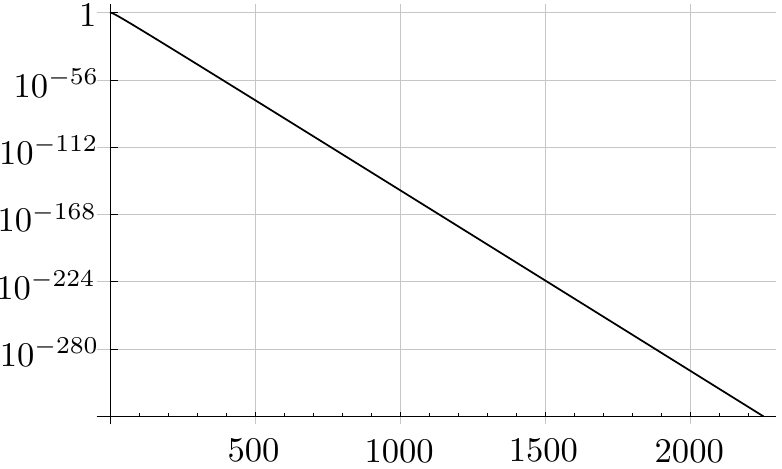}}\hfill
\subfigure[Comparison with \cite{MARSH_RANDGRAV}]{\label{fig:sim-b}\includegraphics[width=0.49\textwidth]{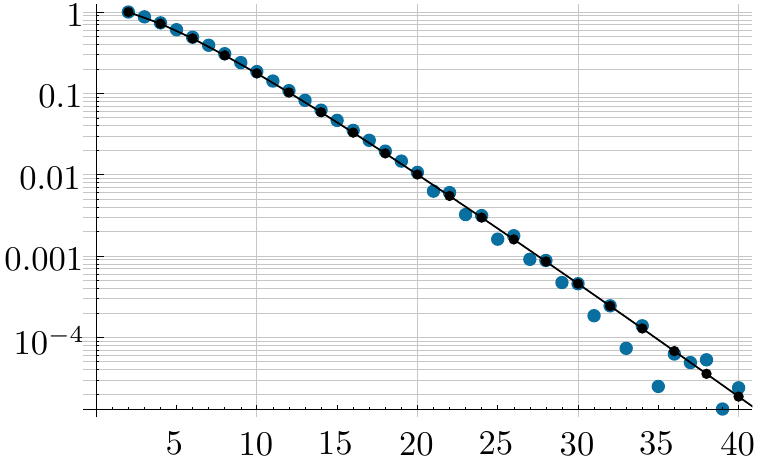}}
\subfigure[Computation time (in s)]{\includegraphics[width=0.49\textwidth]{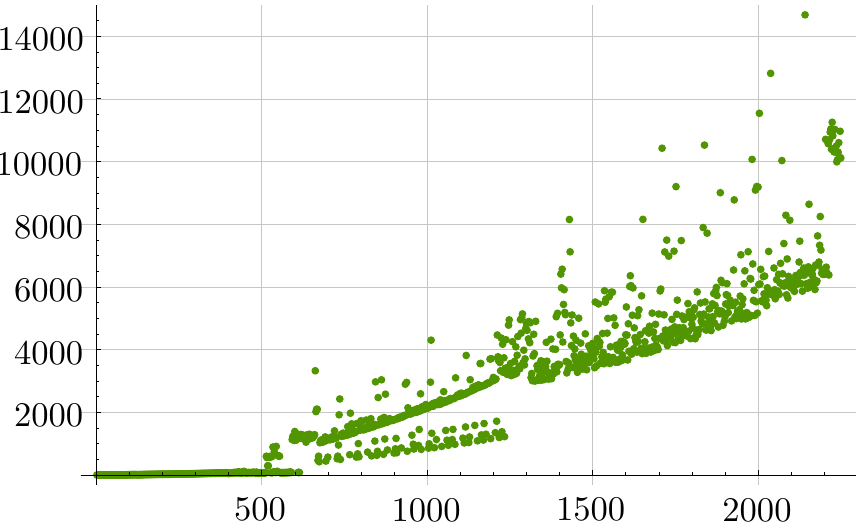}}\hfill
\subfigure[Relative errors]{\label{fig:sim-d}\includegraphics[width=0.49\textwidth]{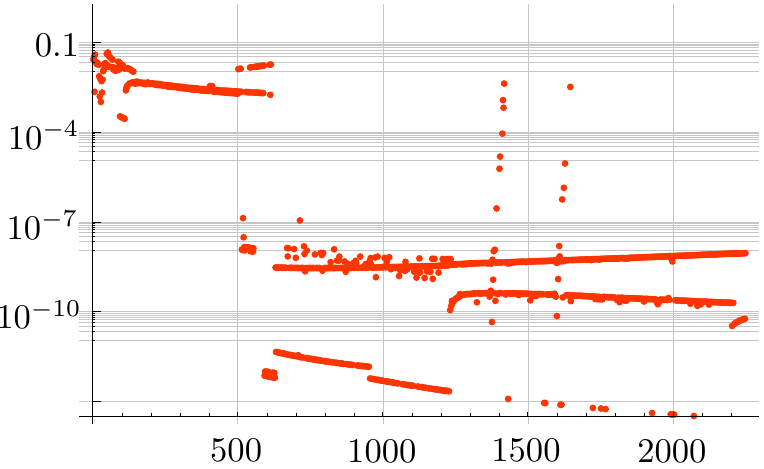}}
\caption{Figure (a) shows the $N$ dependence of $P_N$. Figure (b) shows that the result is coherent with the simulation
of the $\mathcal S$-terms of the mass matrix,
as done by \cite{MARSH_RANDGRAV}. Figure (c) shows the computation time needed to calculate the values and (d) the relative errors, as determined by the algorithm.}
\label{fig:sim}
\end{figure}

Let us return to the issue of finding the $N$ dependence of the fluctuation probability of the mass matrix in our $4d$ random supergravity.
For now accepting the approximation $\mathcal T_-\approx1$, we want to calculate
\begin{equation}
P_N:=\p\left(F^2\sum_{b'=2}^N \frac{|U_{11b'}|^2}{\avg{\lambda_{b'}^2}}\le1\right)=:\p\left(\sum_{i=2}^N a_i X_i\le1\right)=:\p(Y_N\le1)\comma
\end{equation}
where $X_i\sim\chi^2(1)$ and $a_i:=(N\avg{\lambda_{b'}^2})^{-1}$.

Using the method discussed in \prettyref{sec:method}, we calculate the pdf $f_{Y_N}$ for $N$ from $2$ to $2250$ and integrate it over $(0,1)$, which gives $P_N$. All the while the simulation keeps track of the relative error and ensures that it stays below $0.05$. The results can be seen in \prettyref{fig:sim}.

The first $400$ values were calculated on a single i7-2670QM, while the rest were computed on a small cluster with slower individual CPUs (i7-860) and varying load balance, which is where the big amount of noise for the computation time and the relative errors originates.

Let us now return to the question of the distribution of the full mass matrix $m^2_-$. Our first observation is that we can analytically convolve the distributions of $|K^{(3)}|^2$ and $K_{1\bar11\bar1}^{(4)}$. This yields the rather ugly-looking expression
\begin{align*}f_{\mathcal T}(x)=2^{-\frac n4-2} N\ \ee^{-\frac{1}{8} N^2 x^2} &\left(\frac{\sqrt{2} \,
   _1\map F_1\left(\frac{N}{4};\frac{1}{2};\frac{1}{8} (N x-2)^2\right)}{\Gamma
   \left(\frac{N+2}{4}\right)}\right.\\
   &+\left.\frac{(N x-2) \,
   _1\map F_1\left(\frac{N+2}{4};\frac{3}{2};\frac{1}{8} (N x-2)^2\right)}{\Gamma
   \left(\frac{N}{4}\right)}\right)\comma
\end{align*}
which however simplifies gravely for specific values of n. This pdf, alongside the analytical expressions found for the $\mathcal S$-term $f_{Y_N}$, allows us to numerically convolve the two terms and calculate the distribution---and thus the fluctuation probability---of the full mass matrix, for large numbers of $N$ and to a very high precision.

The numerical convolution is done in a straightforward manner using \textsc{Mathematica}. While estimating the error of such a numerical convolution is generally hard, using small enough sample sizes should give us in principle negligible contribution to our error---we mention it for the sake of completeness though. The results can be seen in \prettyref{fig:simSandT}.

\subsection{Data Analysis}

\begin{figure}
\subfigure[Fit to $-c N^p$]{\includegraphics[width=0.49\textwidth]{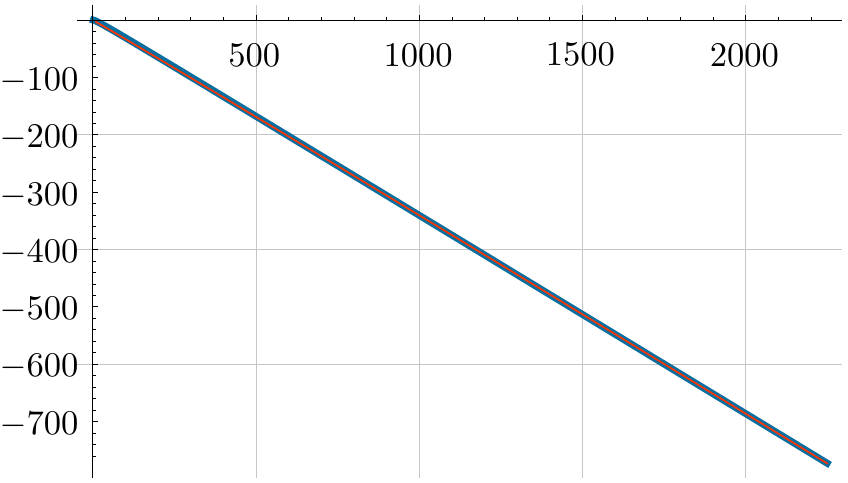}}\hfill
\subfigure[Fit to $\log(a+N)-c-dN$]{\includegraphics[width=0.49\textwidth]{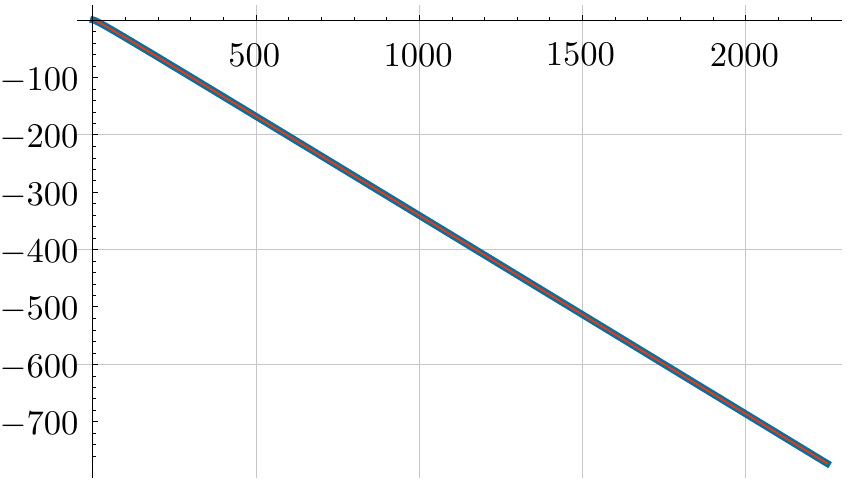}}
\subfigure[Residuals for (a)]{\includegraphics[width=0.49\textwidth]{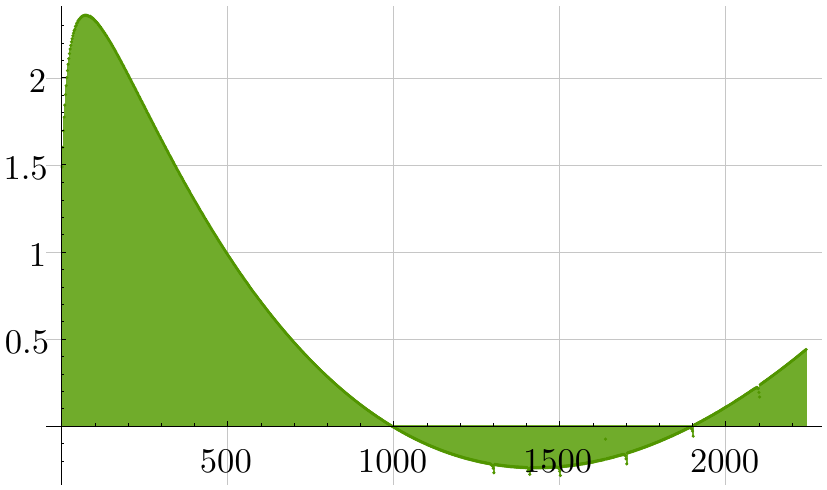}}\hfill
\subfigure[Residuals for (b)]{\label{fig:fit-d}\includegraphics[width=0.49\textwidth]{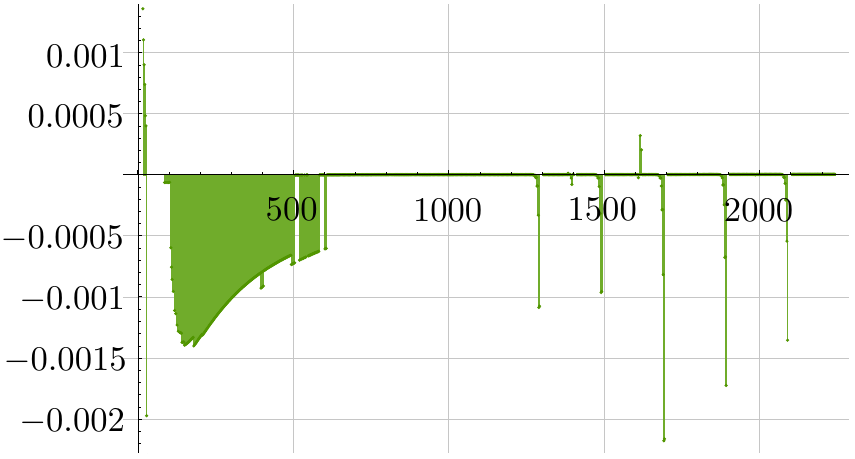}}
\caption{Figure (a) and (b) show plots of fits (orange) to our data (blue) for both proposed models. Figure (c) and (d) show the corresponding residuals.}
\label{fig:fit}
\end{figure}

\begin{figure}
\subfigure[]{\includegraphics[width=0.49\textwidth]{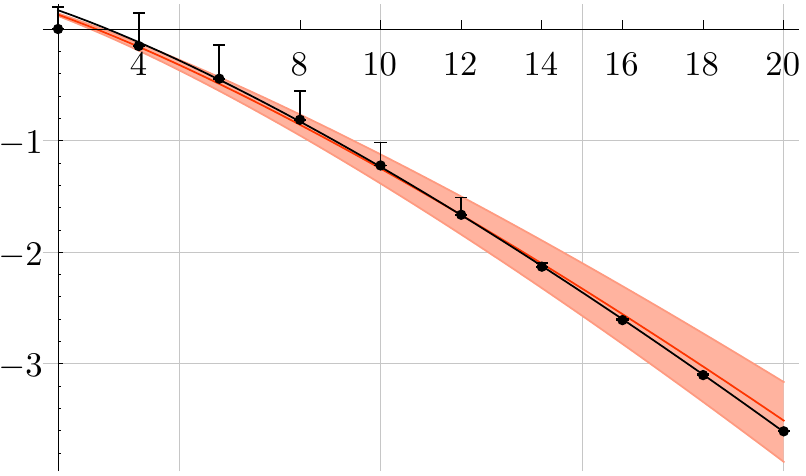}}\hfill
\subfigure[]{\includegraphics[width=0.49\textwidth]{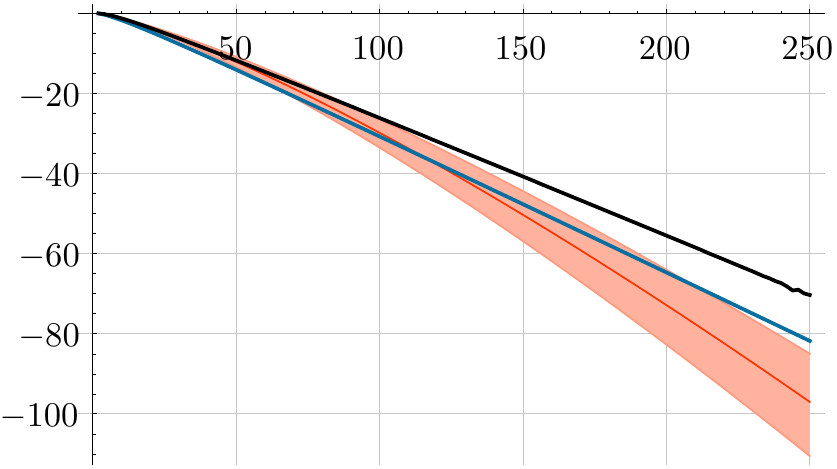}}
\caption{(a) Comparison with \cite[Fig. 7]{MARSH_RANDGRAV}. The shaded area indicates the parameter range for the best fit curve given in \cite{MARSH_RANDGRAV}. The black line is a similar best fit curve (to the values shown) with $p\approx1.31$, so well within the error range. (b) Comparison of $\sset P(\mathcal S\le1)$ (blue), the full $m^2_-$ (black) and the suggested best fit from \cite{MARSH_RANDGRAV}. The error bars for our method are negligibly small.}
\label{fig:simSandT}
\end{figure}

We now turn to the data analysis of our simulation of $\mathcal 
S$ resp. the full mass matrix $m^2_-$.
Our simulation of $\mathcal S$ agrees very well with the numerical simulations of
% of what, this is not the mass matrix
the $\mathcal S$-part of the mass matrix from \cite{MARSH_RANDGRAV}, as can be seen in the low $N$ regime \prettyref{fig:sim-b}. For larger $N$, \prettyref{fig:sim-a}, the overall shape then follows more and more that of a simple exponential, i.e. $\propto\ee^{-cN}$ for some $c$.

Indeed, by fitting the suggested functional dependence $\ee^{-cN^p}$, the residuals show a strong characteristic shape, see \prettyref{fig:fit}. While the overall shape matches pretty well, another function,
\begin{equation}
\ee^{\log(a+N)-c-dN}\equiv(a+N)\ee^{-c-dN}\comma
\end{equation}
matches the shape a lot better, as can be seen by looking at the corresponding residuals. While \prettyref{fig:fit-d} still shows a correlation for the residuals, they are now of the same order as the proposed error, \prettyref{fig:sim-d}. This means that this model is correct, within the given precision.

The fit parameters and standard deviations for both models can be found in \prettyref{tab:fit}.

\begin{table}
\subtable[Parameters for $-cN^p$]{
\begin{tabular}{cll}
\toprule[1pt]
 & Estimate & Standard Error \\ 
\midrule[.5pt]
$c$ & 0.31395 & 0.00011 \\ 
$p$ & 1.01171 & 0.00005 \\ 
\bottomrule[1pt]
\end{tabular}
}\hfill
\subtable[Parameters for $\log(a+N)-c-dN$]{
\begin{tabular}{cll}
\toprule[1pt]
 & Estimate & Standard Error \\ 
\midrule[.5pt]
a & 1.88 & 0.06 \\ 
c & 0.75652 & 0.00011 \\ 
d & 0.34657360 & $4\cdot10^{-8}$ \\ 
\bottomrule[1pt]
\end{tabular}
}\caption{Fit parameters for both proposed models for $\sset P(\mathcal S\le1)$.}
\label{tab:fit}
\end{table}

\begin{table}
\subtable[Parameters for $-cN^p$]{
\begin{tabular}{cll}
\toprule[1pt]
 & Estimate & Standard Error \\ 
\midrule[.5pt]
$c$ & 0.186 & 0.00012 \\ 
$p$ & 1.0759 & 0.00013 \\ 
\bottomrule[1pt]
\end{tabular}
}\hfill
\subtable[Parameters for $\log(a+N)-c-dN$]{
\begin{tabular}{cll}
\toprule[1pt]
 & Estimate & Standard Error \\ 
\midrule[.5pt]
a & 0.4 & 0.6 \\ 
c & 0.615 & 0.013 \\ 
d & 0.30096 & $6\cdot10^{-5}$ \\ 
\bottomrule[1pt]
\end{tabular}
}\caption{Fit parameters for both proposed models for the full mass matrix $m^2_-$.}
\label{tab:fit2}
\end{table}

There is an enormous difference in the amount of data available between the aforementioned methods and our newer one. In our paper---and this agrees with \cite{MARSH_RANDGRAV} and \cite{CHEN_RANDGRAV}---the first $\set O(20)$ data points show a sharp downward bent in a logarithmic plot, suggesting that the probability should go like $P\propto e^{-c_1 N^p}$ with $p\gtrsim1.25$ (and indeed a fit to those first $20$ points yields just such a behaviour).

However---and this again can already be seen as trend in \cite{MARSH_RANDGRAV} and \cite{CHEN_RANDGRAV}---the data points flatten more and more, becoming more and more linear. Unfortunately, the above-mentioned authors did not have sufficiently accurate data for this regime to explore this behaviour. Due to the sheer number of data points in our model ($\set O(2000)$), the behaviour at low $N$ does not receive much weight. Since we are interested in high $N$, though, this should not pose a problem.

For the full mass matrix, for which we added in the $N$-dependence of the $\mathcal T$-term, we once again basically confirm what has been done in \cite{MARSH_RANDGRAV}, at least in the lower $N$ regime that was accessible for them. In \prettyref{fig:simSandT}-a, we see that we can reproduce the previous value of $p\approx1.3$.
Analogously to the case of $\mathcal S$ alone, though, for higher values of $N$, the fluctuation probability for $m^2_-$ to become positive flattens significantly. Repeating the same numerical fit procedure as for the $\mathcal S$-term alone, we obtain the fit parameters in \prettyref{tab:fit2}.

This is the major result of the application of our method to the open problem of a full stability analysis of \knownth{deSitter} vacua in string theory. In brief, it suggests that the probability of finding a metastable \knownth{deSitter} vacuum follows the asymptotic form
\begin{empheq}[box=\emphbox]{equation*}
P_N\sim N\ee^{-0.30096 \pm 6\cdot10^{-5}N}\point
\end{empheq}

\section{Conclusions and Future Work}
The first main result is the analytic expression for the density of a sum of two $\Gamma$ random variables (\prettyref{lem:gammaconv}).
With this expression as a building block, we have derived an efficient algorithm (\prettyref{lem:convalg}) for the calculation of a linear combination of an even number of $\chi^2_r$ random variables.

Unfortunately, at this point, the authors were not able to calculate the numerical convolution of the full mass matrix $m_-^2$ to $N=\set O(2000)$ due to computational restrictions. Qualitatively, though, it is clear that the value for $p$ will most likely converge to something in the range $1\lesssim p\lesssim1.08$.

\begin{em}
With these methods, a reliable estimate for the probability of a metastable \knownth{deSitter} vaccum in a $\mathcal N=1$ $4d$ supergravity has been shown to be of the asymptotic form $N\ee^{-0.3 N}$, where $N$ denotes the number of scalar fields in the theory. As discussed, this last result gives a much weaker bound than the one suggested by \cite{MARSH_RANDGRAV} or \cite{CHEN_RANDGRAV}.
\end{em}

This result is significant. One important implication for the existence of flux vacua is the following. If we assume that there are $M\in\set O(10^{500})$ local extrema, of which $M_+\ll M$ are metastable \knownth{deSitter} vacua, we can conclude that
\[\log(a+N)-c-dN\sim\log N-dN\sim\log\left(\frac{M_+}M\right)\point\]
Consequently, this description of a $4d$ random supergravity is plausible---i.e. $M_+>1$, which means that there is at least one metastable vacuum---if $-500\log 10\lesssim\log N-dN$ or
$N\lesssim 3800$.

\section*{Acknowledgement}I am grateful to Liam McAllister for helpful discussions and for pointing out this project to me. I would also like to thank David Marsh and Timm Wrase for their simulation data and for reviewing my ideas.
Furthermore the author would like to thank the reviewers of this article for their comments that helped to considerably improve the manuscript.

\renewcommand\appendixname{Appendix\ }
\appendix
\section{More Exact Formulae for $\chi^2$ Random Variables}
There are a few more analytic expressions for the pdf $f$ of a sum of $\chi^2$ random variables.
\begin{remark}
Let $X,Y,Z\sim\chi^2_1$ iid random variables and $a,b\in\sset R_{>0}$. Then
\begin{align*}
f_{aX+aY+bZ}(x)&=\theta(x)\sqrt{\frac1{\pi a^2b}}\sqrt{\frac{b-a}{ab}}\ \ee^{-\frac x{2b}}\map F\!\left(\!\!\sqrt{\frac{b-a}{2ab}x}\right)\\
&\equiv\theta(x)\sqrt{\frac1{4a^2b}}\sqrt{\frac{b-a}{ab}}\ \ee^{-\frac x{2a}}\map{erfi}\!\left(\!\!\sqrt{\frac{b-a}{2ab}x}\right)
\comma
\end{align*}
where $\map F$ is the \knownth{Dawson} integral and $\map{erfi}$ the imaginary error function.
\end{remark}
\begin{remark}
Let $X,Y,Z,W\sim\chi^2_1$ iid random variables and $a,b\in\sset R_{>0}$. Then
\begin{align*}
f_{aX+aY+aZ+bW}(x)=\theta(x)\frac14\frac1{\sqrt{a^3b}}\ee^{-\frac{a+b}{4ab}x}\left(\map I_0\left(\frac{a-b}{4ab}x\right)+\map I_1\left(\frac{a-b}{4ab}x\right)\right)\point
\end{align*}
\end{remark}

\section{The Fluctuation Probability of the Full Mass Matrix $m^2_-$}\label{sec:afprob}
For the full mass matrix $m_-^2$, we numerically convolved the $\mathcal S$ and $\mathcal T$-terms and then numerically integrated the probability of getting a positive contribution.

Shown in \prettyref{fig:fprob} are the different probability density functions for the distributions comprising $m^2_-$. It is evident to see that the approximation $\mathcal T\approx1$ is quite crude---the fluctuation probability of the $\mathcal T$-term plays a certain role, even for larger values of $N$.

\begin{figure}
\raisebox{1.4cm}{\makebox[0.5cm][l]{2}}\hfill\subfigure{\includegraphics[width=0.3\textwidth]{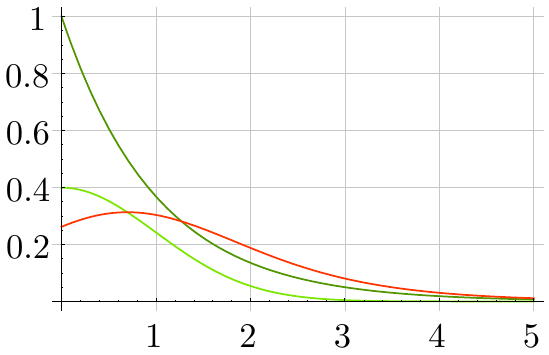}}\hfill
\subfigure{\includegraphics[width=0.3\textwidth]{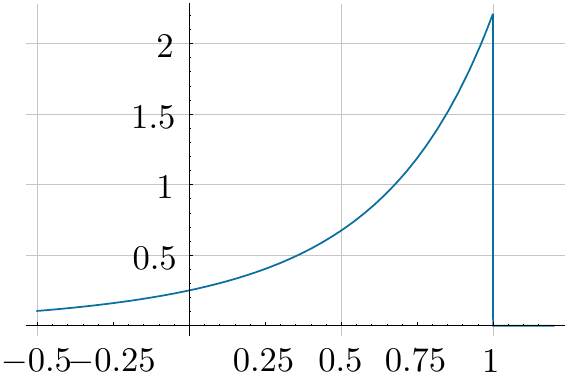}}\hfill
\subfigure{\includegraphics[width=0.3\textwidth]{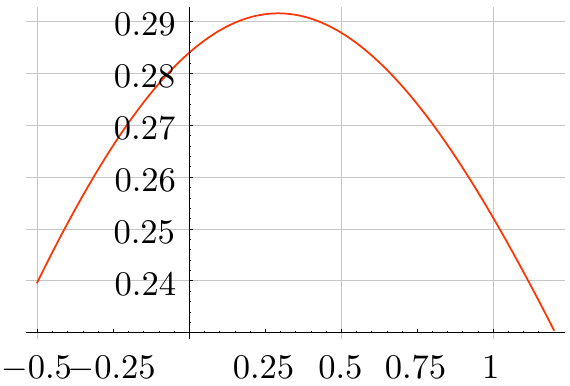}}\\
\raisebox{1.4cm}{\makebox[0.5cm][l]{4}}\hfill\subfigure{\includegraphics[width=0.3\textwidth]{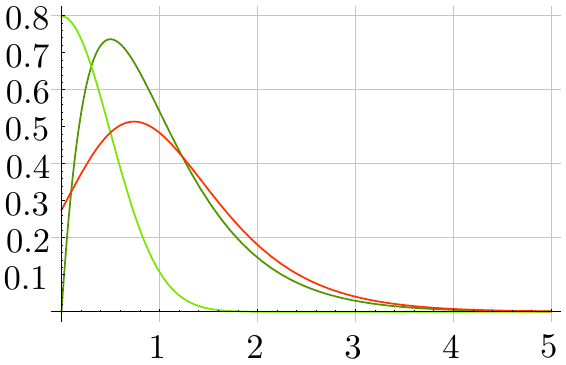}}\hfill
\subfigure{\includegraphics[width=0.3\textwidth]{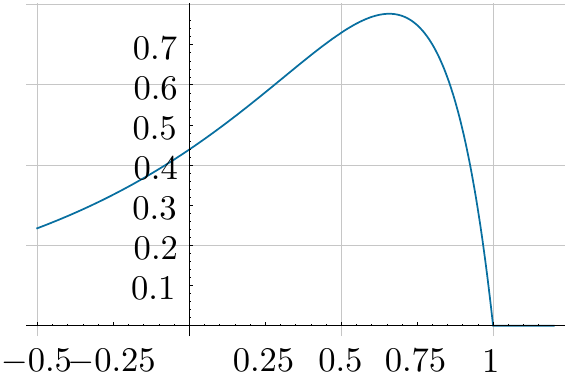}}\hfill
\subfigure{\includegraphics[width=0.3\textwidth]{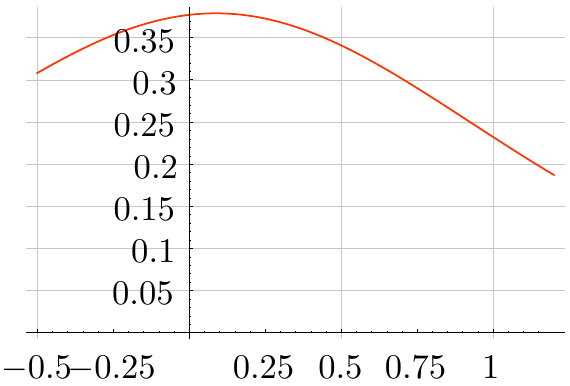}}\\
\raisebox{1.4cm}{\makebox[0.5cm][l]{6}}\hfill\subfigure{\includegraphics[width=0.3\textwidth]{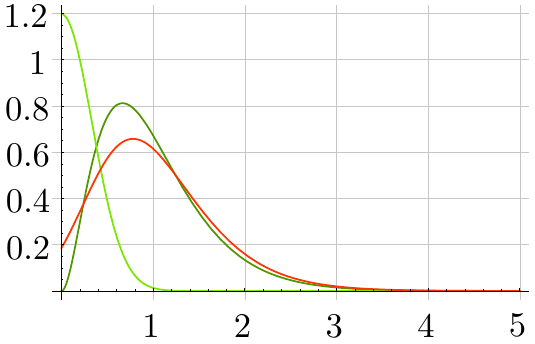}}\hfill
\subfigure{\includegraphics[width=0.3\textwidth]{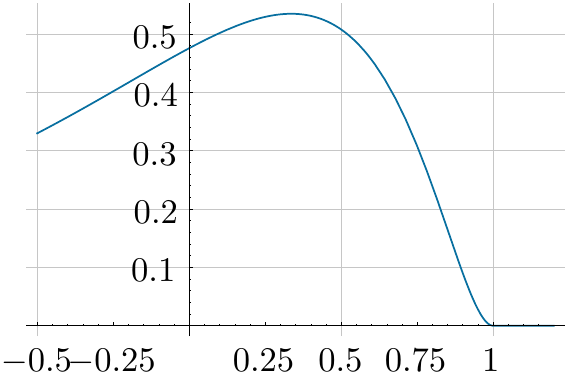}}\hfill
\subfigure{\includegraphics[width=0.3\textwidth]{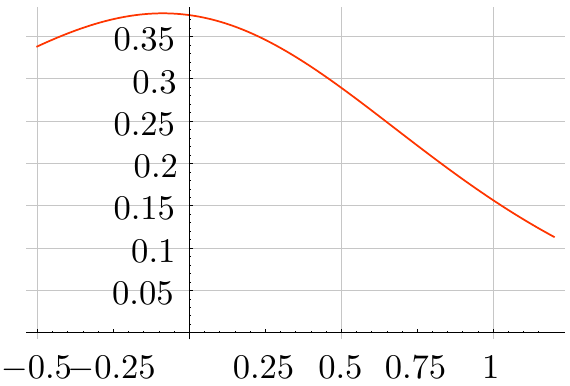}}\\
\raisebox{1.4cm}{\makebox[0.5cm][l]{8}}\hfill\subfigure{\includegraphics[width=0.3\textwidth]{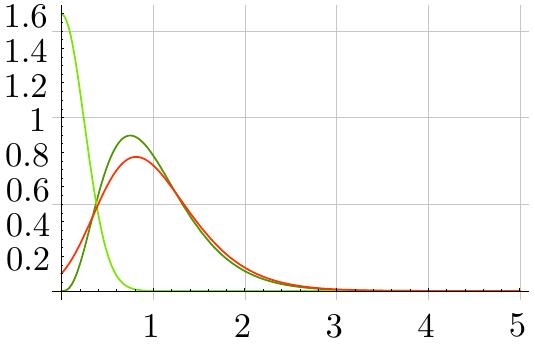}}\hfill
\subfigure{\includegraphics[width=0.3\textwidth]{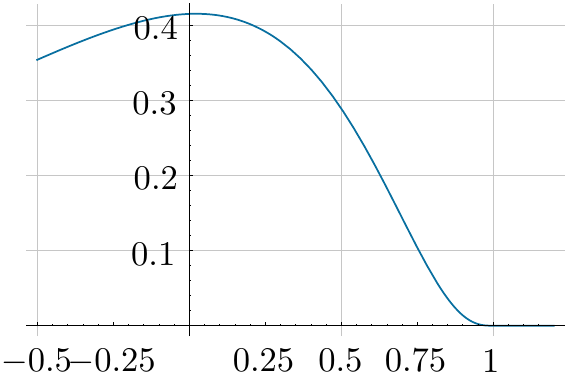}}\hfill
\subfigure{\includegraphics[width=0.3\textwidth]{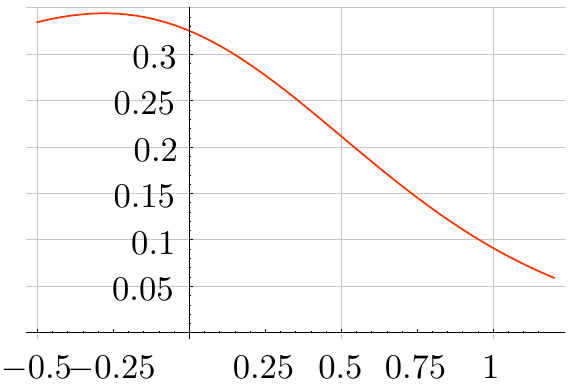}}\\
\[\vdots\]
\raisebox{1.4cm}{\makebox[0.5cm][l]{20}}\hfill\subfigure{\includegraphics[width=0.3\textwidth]{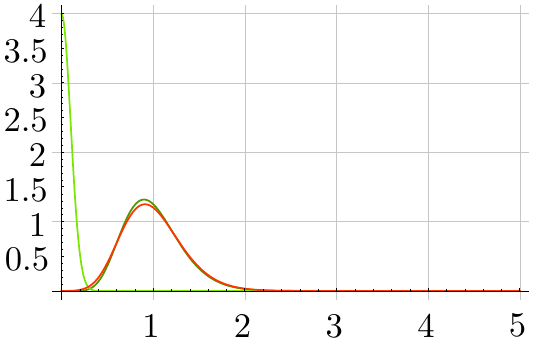}}\hfill
\subfigure{\includegraphics[width=0.3\textwidth]{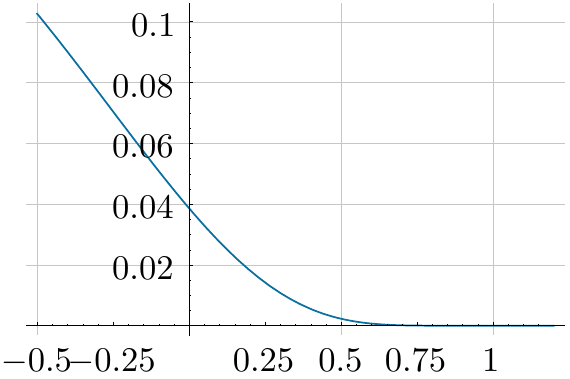}}\hfill
\subfigure{\includegraphics[width=0.3\textwidth]{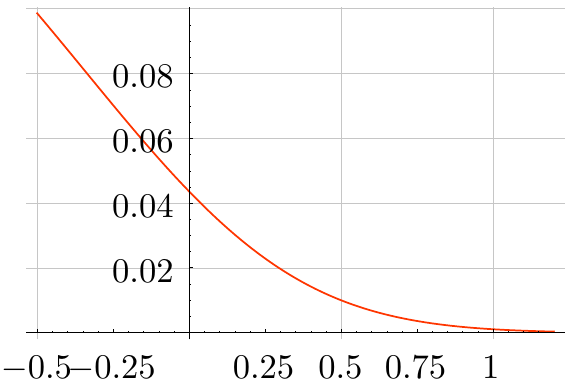}}\\
\[\vdots\]
\raisebox{1.4cm}{\makebox[0.5cm][l]{100}}\hfill\subfigure{\includegraphics[width=0.3\textwidth]{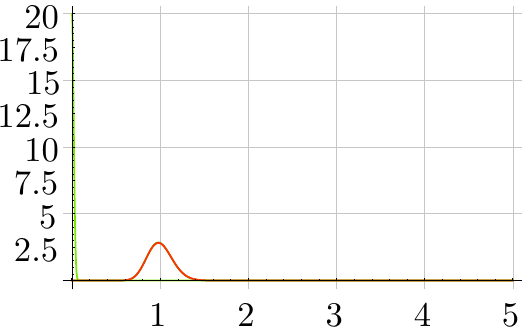}}\hfill
\subfigure{\includegraphics[width=0.3\textwidth]{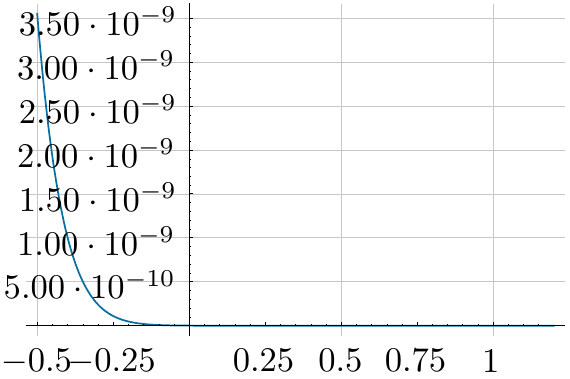}}\hfill
\subfigure{\includegraphics[width=0.3\textwidth]{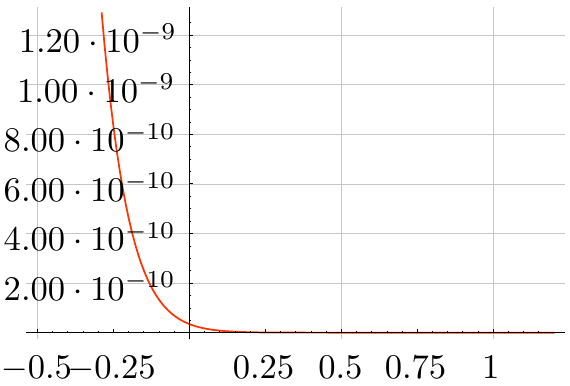}}
\caption{Fluctuation probability for select $N$. Left column: $K_{1\bar11\bar1}$ (light green), $K^{(3)}$ (dark green) and their convolution (red). Middle column: right tail of distribution for $\mathcal S+1$. Right column: right tail of distribution of full mass matrix $m^2_-$.}
\label{fig:fprob}
\end{figure}

% bibliographie
\section{Literature}
\bibliographystyle{babunsrt}
\bibliography{bibliography}

\end{document}

%% file: latex-template/symphony.tex
\usepackage{ifpdf}

\usepackage{microtype}
\usepackage[utf8]{inputenc}
\usepackage{url}
\usepackage{color}
\usepackage{textcomp}
\usepackage{comment}
\usepackage{subfigure}
\usepackage{booktabs}
\usepackage[colorlinks]{hyperref}

\ifpdf
\usepackage{graphicx}
\DeclareGraphicsExtensions{.pdf,.mps,.png,.jpg}
\else
\usepackage{graphicx}
\DeclareGraphicsExtensions{.eps}
\usepackage{pstool}
\fi

\usepackage{mathrsfs}
\usepackage{amsmath}

\usepackage{amsthm}
\usepackage{amssymb}
\usepackage{mathtools}
\usepackage{dsfont}
\usepackage{upgreek}
\usepackage{lmodern}
\usepackage{nicefrac}
\usepackage{empheq}

\definecolor{emphboxcolor}{rgb}{.8, .8, 1}
\def\mystrut(#1,#2){\vrule height #1pt depth #2pt width 0pt} 
\newcommand*\emphbox[1]{\colorbox{emphboxcolor}{\mystrut(20,15)\hspace{3em}#1\hspace{3em}}}

\DeclareMathAlphabet{\mathpzc}{T1}{pzc}{m}{it}

\newtheorem{definition}{Definition}
\newtheorem{theorem}{Theorem}
\newtheorem{corollary}{Corollary}
\newtheorem{remark}{Remark}
\newtheorem{lemma}{Lemma}

\usepackage{slashed}
\usepackage{feynmp}

\usepackage{prettyref}

\newrefformat{sec}{section~\ref{#1}}
\newrefformat{def}{definition~\ref{#1}}
\newrefformat{th}{theorem~\ref{#1}}
\newrefformat{cor}{corollary~\ref{#1}}
\newrefformat{rem}{remark~\ref{#1}}
\newrefformat{lem}{lemma~\ref{#1}}
\newrefformat{ex}{example~\ref{#1}}
\newrefformat{eq}{equation~\ref{#1}}
\newrefformat{fig}{figure~\ref{#1}}
\newrefformat{tab}{table~\ref{#1}}

\usepackage{xspace}
\usepackage{xargs}

\definecolor{orange}{rgb}{.7,0.4,0.0}
\definecolor{darkr}{rgb}{0.5,0.1,0.1}

\newcommand{\set}[1] {\mathrm{#1}}
\newcommand{\sset}[1] {\mathds{#1}}
\newcommand{\map}[1] {\mathrm{#1}}
\newcommand{\op}[1] {\mathbf{#1}}
\newcommand{\sop}[1] {\mathds{#1}}

\newcommand{\avg}[1] {\left< #1\right>}

\newcommand{\point}[0] {\quad.}
\newcommand{\comma}[0] {\quad,}
\newcommand{\mt}[1] {\quad\mathrm{#1}\quad}

\newcommand{\knownth}[1] {\textit{#1}}

\newcommand{\dd}[0] {\mathrm{d}}

\newcommand{\p}[0] {\sop{P}}

\newcommand{\C}[0] {\mathcal C}

\newcommand{\ii}[0] {\mathrm{i}}
\newcommand{\ee}[0] {\mathrm{e}}

\newcommand{\bigsign}[2][1.6] {
\FPset\whelp{#1}  
\FPmul\whelp\whelp{.55}
\FPsub\whelp\whelp{.3}
\raisebox{-\FPprint\whelp em}{\resizebox{#1 em}{!}{$#2$}}
}
\newcommand{\bigast}[2] {\overset{#2}{\underset{#1}{\bigsign{*}}}}

\newcommand{\momentumarrows}{\fmfcmd{%
vardef middir(expr p, ang) =
  dir(angle direction length(p)/2 of p + ang)
enddef;
style_def mom_left(expr p) =
  shrink(.7);
  l = arclength(p);
  drawarrow(subpath(.5-6/l, .5+6/l) of p shifted(4thick*middir(p, 90)));
  endshrink();
enddef;
style_def mom_right(expr p) =
  shrink(.7);
  l = arclength(p);
  drawarrow(subpath(.5-6/l, .5+6/l) of p shifted(4thick*middir(p, -90)));
  endshrink();
enddef;
}}

\usepackage{fp}

\newenvironment{fgbare}[1]{
\FPset\whelp{#1}  
\FPadd\whelp\whelp{2.2}
\begin{minipage}{\FPprint\whelp\unitlength}
\begin{center}
}{
\end{center}
\end{minipage}
}

\newenvironmentx{fg}[5][4=0,5=0]{
\begin{fgbare}{#1}
\begin{fmffile}{#3}
\momentumarrows
\FPset\0{#1}
\FPsub\0\0{#4}
\global\let\cx\0
\FPset\0{#2}
\FPsub\0\0{#5}
\global\let\cy\0
\begin{fmfgraph*}(\FPprint\cx,\FPprint\cy)
}{
\end{fmfgraph*}
\end{fmffile}
\end{fgbare}
}